\documentclass[
final, nomarks
]{dmtcs-episciences}

\usepackage[table]{xcolor}
\usepackage[utf8]{inputenc}
\usepackage{subfigure}

\usepackage[T1]{fontenc}
\usepackage{amssymb,amsmath,amsthm}
\usepackage{mathtools}
\usepackage{tikz}
\usetikzlibrary{matrix,arrows,patterns,calc,positioning}
\usepackage{ifthen}

\newtheorem{theorem}{Theorem}
\newtheorem{lemma}[theorem]{Lemma}
\theoremstyle{definition}

\newcommand\abs[1]{\left|#1\right|}
\newcommand\p[1]{\left(#1\right)}
\newcommand\cp[1]{\left\{#1\right\}}

\author[A.\ F.\ Frankl\'{i}n]{Atli Fannar Frankl\'{i}n}
\title{Pattern avoiding permutations enumerated by inversions}
\affiliation{University of Iceland, Reykjavík, Iceland}

\begin{document}
\publicationdata{vol. 27:1, Permutation Patterns 2024}{2025}{5}{10.46298/dmtcs.14437}{2024-10-11; 2024-10-11; 2025-05-20}{2025-05-21}
\maketitle
\begin{abstract}
\vspace*{0.25cm}
  Permutations are usually enumerated by size, 
  but new results can be found by enumerating them by inversions instead, 
  in which case one must restrict one's attention to indecomposable permutations. 
  In the style of the seminal paper by Simion and Schmidt, we investigate all
  combinations of permutation patterns of length at most $3$.
\end{abstract}

\section{Introduction}

To enumerate a family of permutations by their number of inversions we need the family of permutations with a given number of inversions to be finite. This is generally not the case, so to get around this we use the notion of decomposability. We define the \emph{sum} of two permutations $\rho_1, \dots, \rho_n$ and $\tau_1, \dots, \tau_m$ as a permutation on $n + m$ elements given by 
\[\rho \oplus \tau = \rho_1 \dots \rho_n (n+\tau_1) \dots (n+\tau_m)\]
Following the definitions of Lentin~\cite{lentin} and Comtet~\cite{comtet} we will say that $\pi$ is \emph{decomposable} if $\pi = \rho \oplus \tau$ for some non-empty permutations $\rho$ and $\tau$. In the same vein we call permutations that are not decomposable \emph{indecomposable}. With this definition every permutation can be factored into a set of indecomposable factors, we call these factors its \emph{components}.

An \emph{inversion} in a permutation $\pi = \pi_1 \dots \pi_n$ is a pair of indices $(i, j)$ such that $i < j$ and $\pi_i > \pi_j$. The number of inversions in a permutation $\pi$ will be denoted $\operatorname{inv}(\pi)$. The following lemma is due to Claesson, Jelínek and Steingrímsson and highlights a relation between these two concepts.

\begin{lemma}[\cite{akc}, Lemma 8]\label{invcomp}
Let $\pi$ be a permutation on $n$ elements and $c$ be the number of components of $\pi$. Then $\operatorname{inv}(\pi) \geq n - c$.
\end{lemma}

Crucially, if $c = 1$ then $\operatorname{inv}(\pi) \geq n - 1$ and thus the set of indecomposable permutations with $k$ inversions is finite since it is contained in the set of all permutations on $k + 1$ or fewer elements. With this in mind we define $I_k$ to be the set of all indecomposable permutations with exactly $k$ inversions.

We recall some notions related to patterns in permutations, following the notation summarised by Bevan~\cite{basicnotation}. We will restrict ourselves to only consider classical permutation patterns. We say that two sequences $a_1, \dots, a_m$ and $b_1, \dots, b_m$ are \emph{order-isomorphic} if $a_i < a_j$ holds if and only if $b_i < b_j$. For permutations $\pi = \pi_1 \dots \pi_n$ and $\tau = \tau_1 \dots \tau_m$ we say that $\pi$ contains the pattern $\tau$ if there exists a set of indices $i_1, \dots, i_m$ such that $\pi_{i_1} \dots \pi_{i_m}$ is order-isomorphic to $\tau_1 \dots \tau_m$. We call such a set of indices an \emph{occurrence} of the pattern $\tau$. If $\pi$ has no occurrence of $\tau$ we will say that $\pi$ \emph{avoids} the pattern $\tau$. The subset of $I_k$ consisting of the permutations avoiding $\tau$ by will be denoted $I_k(\tau)$. Similarly, we will denote the subset of $I_k$ consisting of permutations avoiding several patterns $\tau_1, \tau_2, \dots, \tau_r$ by $I_k(\tau_1, \tau_2, \dots, \tau_r)$.

In this paper we will investigate all combinations of patterns of length $\leq 3$ in the style of Simion and Schmidt's~\cite{seminal} paper. Something similar in the case of indecomposable permutations enumerated by size (rather than inversions) has been done before by Gao, Kitaev and Zhang~\cite{indecomp}, where they determine the number of indecomposable permutations for all patterns of length $\leq 4$ in terms of the number of decomposable permutations.

Another paper in this area is due to Dokos, Dwyer, Johnson, Sagan and Selsor~\cite{dokos}. That paper finds $q$-polynomials for permutations enumerated jointly by size and inversions. They find many algebraic connections related to polynomials and $q$-analogues, covering analogues of the Catalan numbers, Fibonacci numbers, triangular numbers and powers of two. We deviate from this approach of joint enumeration for the two reasons outlined below. Our results do not follow directly from theirs either as they do not consider all the cases of this paper and in the cases where they do the indecomposability criterion prevents direct derivation.

The first reason is to get novel bijective maps to other combinatorial objects, rather than algebraic equivalences. The motivating example for this paper came from Claesson, Jelínek and Steingrímsson~\cite{akc}. In Lemma 10 of their paper they note that the inversion tables of permutations are weakly decreasing if and only if the permutations avoids the pattern $132$. Adding indecomposability as a condition turns this connection into a bijective map providing the first instance of the type of bijections studied in this paper.

The other reason was to possibly obtain bounds on the growth rates of pattern avoiding permutations. Let $S_n(\pi_1, \dots, \pi_r)$ be the set of permutations of size $n$ avoiding the patterns $\pi_1, \dots, \pi_n$. Define $S_n^*(\pi_1, \dots, \pi_r)$ similarly for indecomposable permutations. The generating function for $S_n^*$ can often be expressed in terms of the generating function of $S_n$, as is done for many different generating functions in Gao et.\ al.~\cite{indecomp}. Thus if we can bound the growth rate of $|S^*_n|$ we can often get bounds for $|S_n|$ as well. Clearly we have
\[S_n^*(\pi_1, \dots, \pi_r) \subseteq \bigcup_{k=0}^{\binom{n}{2}} I_k(\pi_1, \dots, \pi_r)\]

Consider this bound for the pattern $132$. The number of partitions of $n$ grows like $\exp\p{\pi \sqrt{2n/3}}$; see for instance Apostol~\cite[p.~316--318]{apostol}. The bound above would then give us that the growth rate of $|S_n^*(132)|$ is at most $\exp \p{ n \pi \sqrt{1/3}} \approx 6.13^n$. By Gao et.\ al.~\cite{indecomp} we know $|S_n^*(132)|$ grows like the Catalan numbers, so the actual growth rate is $4^n n^{-5/2}$.

Some other cases seem to give tighter bounds, for example $I_k(132, 4321)$. We conjecture based on numerical evidence that this is enumerated by the number of partitions with no parts strictly between the smallest and the largest part. From Bloom and McNew~\cite{bloom} we get $\sum_{k=1}^n |I_k(132, 4321)|$ is asymptotically bounded by $n^2 \log(n)^2$, so $S_n^*(132, 4321)$ grows at most like $n^4 \log(n)^2$. From Albert, Bean, Claesson, Nadeau, Pantone and Úlfarsson~\cite{bean} we get that $S_n(132, 4321)$ grows like $n^4$ and experimental results suggest $S_n^*(132, 4321)$ grows like $n^3$, so the bound is much closer in this case. Another case is $I_k(321, 1342)$, which we conjecture to have $k(k+1)/2+1$ elements. Thus would asymptotically bound $S_n^*(321, 1342)$ as $n^4/8$. From Albert et.\ al.~\cite{bean} we know that $S_n(321, 1342)$ grows like $2^n$, which shows that almost all permutations avoiding $321, 1342$ are decomposable. Experimental results suggest that $S^*_n(321, 1342)$ grows like $n^3/6$ which would make the bound quite good.

A summary of the results in this paper is given in the following table. We omit sets of patterns that give the same sequence as a subset of those patterns. We also omit sequences that decay to zero.

\begin{center}
\rowcolors{2}{gray!25}{white}
\begin{tabular}{|l|l|}
\hline
\rowcolor{gray!50}
Pattern set & Result \\
\hline
123 & No OEIS entry, recurrence in Theorem~\ref{123rec} \\
132 & \texttt{A000041}, Partitions \\
231 & \texttt{A005169}, Fountains \\
321 & \texttt{A006958}, Parallelogram polyominoes \\
123, 231 & No OEIS entry, generating function in Theorem~\ref{fountainsum}  \\
132, 123 & \texttt{A135278}, Pascal triangle with first column removed \\
132, 213 & \texttt{A117629}, Gorenstein partitions \\
132, 231 & \texttt{A000009}, Partitions into distinct parts \\
132, 321 & \texttt{A000005}, Partitions into equal parts, i.e. divisors \\
231, 312 & \texttt{A010054}, Characteristic of triangular numbers \\
231, 321 & \texttt{A000012}, All ones \\
123, 132, 213 & No OEIS entry, generating function in Theorem~\ref{diagonalpascal}  \\
123, 132, 231 & \texttt{A000012}, All ones \\
132, 213, 231 & \texttt{A001227}, Odd divisors \\
123, 132, 213, 231 & \texttt{A103451}, Pascal triangle with values $> 1$ zeroed out \\
\hline
\end{tabular}
\end{center}

The next section will consider single patterns, and introduce many of the fundamental objects we will be mapping to, such as partitions, fountains of coins and parallelogram polyominoes. In section $3$ after we start tackling pairs of patterns, which introduces many refinements of previous maps. We obtain maps to many different types of partitions for example, almost triangular partitions and Gorenstein partitions among them. Lastly in section $4$ all remaining sets of patterns are tackled.

\section{Single patterns}

The sets $I_k(1)$ and $I_k(21)$ have no non-empty elements. Furthermore $|I_k(12)|$ is the characteristic function of the triangular numbers, listed in the OEIS~\cite{oeis} as \texttt{A010054}. This is because to avoid $12$ the permutation must be decreasing and the decreasing permutation of length $n$ has $\binom{n}{2}$ inversions. The single patterns of length $3$ are $123, 132, 213, 231, 312$ and $321$. To reduce the number of cases we consider symmetries on $I_k$, one of them being the \emph{reverse complement} of a permutation, written out explicitly this is 
\[\pi_1 \dots \pi_{n-1} \pi_n \mapsto (n + 1 - \pi_n)(n + 1 - \pi_{n-1})\dots(n + 1 - \pi_1)\]
The number of inversions remains constant under reverse complement. Consider an inversion $i < j$ in $\pi$. The values $\pi_i > \pi_j$ get mapped to $n + 1 - \pi_i, n + 1 - \pi_j$ at indices $n + 1 - i, n + 1 - j$ in the image, so it cancels out. We also see that $\pi$ is decomposable if and only if its reverse complement is, so the reverse complement is an involution on $I_k$ for every $k$. It maps the pattern $132$ to the pattern $213$ and from this we see that $|I_k(213)| = |I_k(132)|$ so we only have to consider one of these patterns. In our case the only symmetries are inverse, reverse complement and the composition of those two maps. This is in contrast to when enumerating by size where both reverse and complement are symmetries unto themselves.

Using these symmetries we see that it suffices to consider $123$, $132$, $231$ and $321$. We start with $132$ and we will make use of a well known bijection. The \emph{inversion table} of a permutation $\pi = \pi_1 \dots \pi_n$ is the sequence $b_1b_2 \dots b_n$ where $b_i$ is the number of values after $\pi_i$ in $\pi$ that are smaller than $\pi_i$. The image under the map $\pi \mapsto b_1 \dots b_n$ is the set of all sequences such that the first value is $\leq n - 1$, the next $\leq n - 2$ and so on. We call such sequences \emph{subdiagonal}. Furthermore, we call elements in a subdiagonal sequence that are equal to their maximum possible value \emph{diagonal} elements.

First we recall a lemma due to Claesson, Jelínek and Steingrímsson that will be used a few times going forward. 

\begin{lemma}[Lemma 10,~\cite{akc}]\label{inv132}
	The inversion table of a permutation is weakly decreasing if and only if the permutation avoids $132$.
\end{lemma}

Now we simply have to formally prove that adding the indecomposability criterion turns this lemma into a bijective map.

\begin{theorem}\label{partitions}
$|I_k(132)|$ is equal to the number of partitions of $k$, listed as \texttt{A000041} in the OEIS.
\end{theorem}
\begin{proof}
Let $\pi \in I_k(132)$ and consider its inversion table, it is weakly decreasing by Lemma~\ref{inv132}. The inversion table may have trailing zeroes, so let $\rho_1, \dots, \rho_r$ be the values in the inversion table once the trailing zeroes have been removed. Clearly, the sum of the $\rho_i$ is equal to $k$ and hence $\rho$ is a partition of $k$. We claim that $\pi \mapsto \rho$ is a bijection between $I_k(132)$ and the set of partitions of $k$.

We start by showing that the map $\pi \mapsto \rho$ is injective. As the mapping from permutations to inversion tables is bijective, we only have to establish that dropping the trailing zeroes in the inversion table is an injective operation. Suppose that we have two permutations $\pi$ and $\tau$ that have the same inversion table up to trailing zeroes. Suppose $\tau$ has $\ell$ extra trailing zeroes compared to $\pi$. Then we must have $\tau = \pi \oplus (1 \oplus 1 \oplus \dots \oplus 1)$ where the $1$ has been repeated $\ell$ times. But then $\tau$ is decomposable, so it is not in our domain. Thus the map is injective on the domain.

Finally we show surjectivity. Let  $\rho = (\rho_1, \dots, \rho_r)$ be a partition. We can append zeroes to $\rho$ until it becomes subdiagonal. Let the total number of entries after appending zeroes be $n$, so $\rho = (\rho_1, \dots, \rho_n)$. Now let $d \coloneqq \max_{i=1}^r \rho_i  - (n - i)$, because $\rho$ is subdiagonal we have $d \leq 0$. We have $\rho_r \geq 1$ so if $d < 0$ then $n - r \geq 2$. Therefore if $d < 0$ we could have omitted one of the trailing zeroes and still maintained subdiagonality, so this is not the case. Then there must be some index $i \leq r$ that achieves the maximum $d = 0$, so $\rho_i = n - i > 0$. Treating $\rho$ as an inversion table, let the corresponding permutation be $\pi = \pi_1 \dots \pi_n$. We prove that $\pi$ is indecomposable by contradiction, so assume the first $j < n$ elements of $\pi$ form a component. The last value in a component is smaller than every value after it, so we have $\rho_j = 0$. Furthermore, $(\rho_1, \dots, \rho_j)$ is subdiagonal and $\rho_j$ must be followed by only zeroes as $\rho$ is weakly decreasing. But this contradicts the existence of the index $i$, so $\pi$ must be indecomposable.
\end{proof}

To work with $231$-avoiding permutations we define the \emph{skew-sum} of two permutations $\pi_1 \dots \pi_n$ and $\tau_1 \dots \tau_m$ as a permutation on $n + m$ elements given by $\pi \ominus \tau = (\pi_1 + m) \dots (\pi_n + m) \tau_1 \dots \tau_m$. We will say that $\pi$ is \emph{skew-decomposable} if $\pi = \tau_1 \ominus \tau_2$ for non-empty permutations $\tau_1$ and $\tau_2$. If $\pi$ is not skew-decomposable we call $\pi$ \emph{skew-indecomposable}.

A \emph{fountain of coins} is an arrangement of coins in rows such that the bottom row is full (that is, there are no ``holes''), and such that each coin in a higher row rests on two coins in the row below~\cite{origfountain}, see example in Figure~\ref{231ex}. The connection between the inversions of $231$-avoiding permutations and fountains is not new, it is due to Br\"{a}nd\'{e}n, Claesson and Steingr\'{i}msson~\cite{andersfountain}. But while their result is very similar to ours, we need to incorporate indecomposability into our argument, so we need a separate proof.

\begin{theorem}\label{ik231}
$|I_k(231)|$ is equal to the number of fountains on $k$ coins. The corresponding sequence is entry \texttt{A005169} in the OEIS.
\end{theorem}

\begin{figure}
	\centering
	\begin{tikzpicture}
\node[circle,draw,minimum size=1cm](r1c1) at (0,0) {};
\node[circle,draw,minimum size=1cm, red, dashed, thick](r2c1) at (60:1cm) {};
\node[circle,draw,minimum size=1cm, red, dashed, thick](r3c1) at (60:2cm) {};
\node[circle,draw,minimum size=1cm, red, dashed, thick](r4c1) at (60:3cm) {};
\node[circle,draw,minimum size=1cm, red, dashed, thick](r5c1) at (60:4cm) {};
\node[circle,draw,minimum size=1cm, red, dashed, thick](r6c1) at (60:5cm) {};

\node[circle,draw,minimum size=1cm](r1c2) at ($(r1c1) + (0:1cm)$) {};
\node[circle,draw,minimum size=1cm](r2c2) at ($(r2c1) + (0:1cm)$) {};
\node[circle,draw,minimum size=1cm](r3c2) at ($(r3c1) + (0:1cm)$) {};
\node[circle,draw,minimum size=1cm, red, dashed, thick](r4c2) at ($(r4c1) + (0:1cm)$) {};
\node[circle,draw,minimum size=1cm, red, dashed, thick](r5c2) at ($(r5c1) + (0:1cm)$) {};

\node[circle,draw,minimum size=1cm](r1c3) at ($(r1c2) + (0:1cm)$) {};
\node[circle,draw,minimum size=1cm](r2c3) at ($(r2c2) + (0:1cm)$) {};
\node[circle,draw,minimum size=1cm, red, dashed, thick](r3c3) at ($(r3c2) + (0:1cm)$) {};
\node[circle,draw,minimum size=1cm, red, dashed, thick](r4c3) at ($(r4c2) + (0:1cm)$) {};

\node[circle,draw,minimum size=1cm](r1c4) at ($(r1c3) + (0:1cm)$) {};
\node[circle,draw,minimum size=1cm, red, dashed, thick](r2c4) at ($(r2c3) + (0:1cm)$) {};
\node[circle,draw,minimum size=1cm, red, dashed, thick](r3c4) at ($(r3c3) + (0:1cm)$) {};

\node[circle,draw,minimum size=1cm](r1c5) at ($(r1c4) + (0:1cm)$) {};
\node[circle,draw,minimum size=1cm](r2c5) at ($(r2c4) + (0:1cm)$) {};

\node[circle,draw,minimum size=1cm](r1c6) at ($(r1c5) + (0:1cm)$) {};

\node[red] (t1) at ($(r6c1) + (60:0.75cm)$) {\Large $5$};
\node[red] (t2) at ($(r5c2) + (60:0.75cm)$) {\Large $2$};
\node[red] (t3) at ($(r4c3) + (60:0.75cm)$) {\Large $2$};
\node[red] (t4) at ($(r3c4) + (60:0.75cm)$) {\Large $2$};
\node[red] (t5) at ($(r2c5) + (60:0.75cm)$) {\Large $0$};
\node[red] (t6) at ($(r1c6) + (60:0.75cm)$) {\Large $0$};

\node (m1) at (-4, 0) {\Large $\rho' = 522200$};
\node (m2) at (-4, 2) {\Large $\pi' = 634512$};
\node (m3) at (-4, 4) {\Large $\pi = 215436$};

\draw[<->, thick] (m1) -- (m2) node[midway,right] () {Inversion table};
\draw[<->, thick] (m2) -- (m3) node[midway,right] () {Reverse};
\draw[<->, thick] (m1) -- (-0.7, 0) node[midway,above] () {Fountain};
	\end{tikzpicture}
	\caption{Example of bijective map in Theorem~\ref{ik231}.}
	\label{231ex}
\end{figure}

\begin{proof}

Let $\pi'$ be the reverse of $\pi \in I_k(231)$ and $n$ be the number of elements of $\pi$, this means every pair of indices in $\pi$ is an inversion if and only if that pair is not an inversion in $\pi'$. Therefore $\pi'$ has $\binom{n}{2} - k$ inversions. We note that the indecomposability of $\pi$ is equivalent to the skew-indecomposability of $\pi'$. Since $\pi'$ avoids $132$ we get a partition $\rho'$ of $\binom{n}{2} - k$ in the same manner as in Theorem~\ref{partitions}. Note, however, that since $\pi'$ might be decomposable we cannot recover the value of $n$ from $\rho'$ alone. We know $\rho_i' \leq n - i$ since $\rho_i$ is given by the entries of an inversion table, which is subdiagonal. Suppose $\rho_i' = n - i$ for some $i$, since $\rho'$ is weakly decreasing this means $\rho_j' \geq n - i$ for $j \leq i$. Then the first $i$ elements in $\pi'$ dominate the last $n - i$, meaning that $\pi'$ is skew-decomposable. This contradicts the fact that $\pi$ is indecomposable, so we must have $\rho_i' < n - i$ for all $i$. 

We now define a bijective map from $I_k(132)$ to fountains with $k$ coins. Consider a full fountain of coins, meaning it has $n - 1$ coins in the bottom row, $n - 2$ coins in the one above and so on, so $\binom{n}{2}$ coins in total. Let us consider diagonals going from a coin in the bottom row up and to the right. The $i$-th diagonal has $n - i$ coins. Since $\rho_i' < n - i$ we can remove the last $\rho_i'$ coins from the $i$-th diagonal without altering the bottom row. This leaves us with an arrangement of $\binom{n}{2} - (\binom{n}{2} - k) = k$ coins. Let us show that this results in a fountain. The only way it could fail to is if we do not remove some coin $c$, but remove a coin out from under it. The coin below and to the left of $c$ is in the same diagonal, and since we remove coins starting from the end, we never remove this coin before $c$. Thus we consider the coin below and to the right of $c$. It is in the diagonal after the one $c$ is in. For it to be removed we would have to remove strictly more elements from that diagonal than from the one $c$ is in. But $\rho'$ is a partition and is thus weakly decreasing, so this cannot happen. Thus a $231$-avoiding permutation produces a fountain on $k$ coins.

To show bijectivity we give the inverse. Let a fountain on $k$ coins with $n - 1$ coins in the bottom row be given. We can read out the number of coins missing from each diagonal to get a sequence of numbers. By the same argument as above this sequence is weakly decreasing. Therefore, this gives us a valid partition $\rho'$ of $\binom{n}{2} - k$. Since we know $n$ from the number of coins in the bottom row we can recover $\pi'$ uniquely. Since $\rho_i$ is the number of coins missing from diagonal $i$ we must have $\rho_i' < n - i$. Thus the first $i$ elements of $\pi'$ can never dominate the last $n - i$, so $\pi'$ is skew-indecomposable. Reversing $\pi'$ to get $\pi$ we then get an indecomposable permutation. Since $\pi'$ is constructed from a partition it avoids $132$ by Lemma~\ref{inv132}, so $\pi$ avoids $231$, completing our proof.
\end{proof}

Enumerating $321$-avoiding permutations by size and inversions (allowing decomposable permutations) has been investigated before by Barcucci, Del Lungo, Pergola and Pinzani~\cite{321}. They derive a generating function, but we will need a bijective proof for a later result, so we give a different proof. Our bijective map will map our permitations to \emph{parallelogram polyominoes}. Following the definition given by Bender~\cite{bender} we can define parallelogram polyominoes in terms of two weakly increasing non-negative sequences $(\ell_1, \ell_2, \dots, \ell_n)$ and $(r_1, r_2, \dots, r_n)$ satisfying $\ell_1 = 0$ and $\ell_i, \ell_{i+1} < r_i$ for all $i$. The polyomino itself can be viewed as consisting of cells in the segment $[\ell_i, r_i]$ at height $i$, see Figure~\ref{parallel}.

\begin{figure}
\centering
\begin{tikzpicture}
\foreach \x/\y in {0/0, 1/0, 2/0, 1/1, 2/1, 3/1, 1/2, 2/2, 3/2, 3/3, 3/4, 4/4, 5/4} {
	\draw[fill=lightgray] (\x, \y) -- (\x + 1, \y)  -- (\x + 1, \y + 1) -- (\x, \y + 1) -- cycle;
}
\draw[thick, ->] (-0.5, 0) -- (6.5, 0);
\draw[thick, ->] (0, -0.5) -- (0, 5.5);
\node (L) at (-3, 3)  {\Large $\ell = (0, 1, 1, 3, 3)$};
\node (R) at (-3, 2)  {\Large $r = (3, 4, 4, 4, 6)$};

\node (l1) at (0.25, 0.25) {$0$};

\draw[solid, <->] (0, 1.5) -- (1, 1.5);
\node (l2) at (0.5, 1.25) {$1$};

\draw[solid, <->] (0, 2.5) -- (1, 2.5);
\node (l3) at (0.5, 2.25) {$1$};

\draw[solid, <->] (0, 3.5) -- (3, 3.5);
\node (l4) at (0.5, 3.25) {$3$};

\draw[solid, <->] (0, 4.5) -- (3, 4.5);
\node (l5) at (0.5, 4.25) {$3$};

\draw[dashed, <->] (0, 0.75) -- (3, 0.75);
\node (r1) at (2.5, 0.5) {$3$};

\draw[dashed, <->] (0, 1.75) -- (4, 1.75);
\node (r2) at (3.5, 1.5) {$4$};

\draw[dashed, <->] (0, 2.75) -- (4, 2.75);
\node (r3) at (3.5, 2.5) {$4$};

\draw[dashed, <->] (0, 3.75) -- (4, 3.75);
\node (r4) at (3.5, 3.5) {$4$};

\draw[dashed, <->] (0, 4.75) -- (6, 4.75);
\node (r5) at (5.5, 4.5) {$6$};
\end{tikzpicture}
\caption{Example of parallelogram polyomino.}
\label{parallel}
\end{figure}

\begin{theorem}\label{ik321}
$|I_k(321)|$ equals the number of parallelogram polyominoes with $k$ cells, listed as \texttt{A006958} in the OEIS.
\end{theorem}
\begin{proof}
We call $\pi_k$ a \emph{left-to-right maximum} of $\pi = \pi_1 \dots \pi_n$ if $\pi_i < \pi_k$ for all $i < k$.  Suppose $\pi$ has $s$ left-to-right maxima at indices $i_1, \dots, i_s$. Consider the values in $\pi$ that are not left-to-right maxima, let us call these values \emph{short} going forward. If any two short values form an inversion, there must be a left-to-right maximum left of them that forms an occurrence of $321$. Thus the short values of $\pi$ are in increasing order, so $\pi$ is uniquely determined by $i_1, \dots, i_s$ and $\pi_{i_1}, \dots, \pi_{i_s}$. We can also recover $n$ since $\pi_{i_s} = n$. 

We now show that taking $\ell \coloneqq (i_1 - 1, i_2 - 2, \dots, i_s- s)$ and $r \coloneqq (\pi_{i_1} - 1, \dots, \pi_{i_s} - s)$ we get a parallelogram polyomino. The sequence $i_1, \dots, i_s$ is strictly increasing, so $\ell$ is weakly increasing. Similarly since $\pi_{i_1}, \dots, \pi_{i_s}$ are left-to-right maxima they are strictly increasing, so $r$ is weakly increasing. The first value of a permutation must be a left-to-right maximum, so $i_1 = 1$ and so $\ell$ starts with a zero and is therefore also non-negative. Next we note that because $\pi_{i_j}$ is a left-to-right maximum it must be greater than the $i_j - 1$ different values before it, so $\pi_{i_j} \geq i_j$. But if $\pi_{i_j} = i_j$ then the first $i_j$ values  of $\pi$ are a permutation unto themselves, making $\pi$ decomposable. Thus $i_j < \pi_{i_j}$ for every $j$, so we get $\ell_j < r_j$ by substracting $j$ from both sides. This leaves only the condition $\ell_{j+1} < r_j$, which simplifies to $i_{j+1} \leq \pi_{i_j}$. This follows similarly since the $\pi_{i_j}$ is the largest element out of the $i_{j+1} - 1$ elements preceding $\pi_{i_{j+1}}$ and they can't exactly be the elements $1$ through $i_{j+1} - 1$ due to indecomposability.

Starting with a parallelogram polyomino we can trace the same steps backward to obtain a permutation in $I_k(321)$ by the same arguments.

All there is left to prove is that this map takes permutations with $k$ inversions to polyominos with $k$ cells. The number of cells in the polyomino defined by $(i_1 - 1, \dots, i_s - s)$ and $(\pi_{i_1} - 1, \dots, \pi_{i_s} - s)$ is
\[\pi_{i_1}  - 1 + \dots + \pi_{i_s} - s - (i_1 - 1) - \dots - (i_s - s) = \pi_{i_1} - i_1 + \dots + \pi_{i_s} - i_s\]
Since the short values and left-to-right maxima are both in internally increasing order, the inversions are all between left-to-right maxima and short values. Because $\pi_{i_j}$ is greater than the $i_j - 1$ elements before it and greater than $\pi_{i_j} - 1$ elements in total, it must be greater than $\pi_{i_j} - i_j$ elements that come after it. All inversions are of this form, so summing over these values gets us the total number of inversions, completing our proof.
\end{proof}

Interpreting parallelogram polyominoes as permutations yields a formula for efficiently computing new terms of \texttt{A006958}, taking $\mathcal{O}(k^2)$ time and space to compute $|I_k(321)|$. We need a small lemma first.

\begin{lemma}\label{321inv}
A sequence $\rho$ of non-negative integers is the inversion table of a permutation in $I_k(321)$ (after adding the minimal required number of trailing zeroes) for some $k$ if and only if the following conditions all hold. Whenever $\rho_i > 0$ and $\rho_j > 0$ with $i < j$ are separated by $\ell$ zeroes then $\rho_j \geq \rho_i - \ell$. Furthermore, $\rho_i > 0$ can be followed by at most $\rho_i - 1$ zeroes, aside from trailing zeroes, and $\rho_1 \neq 0$.
\end{lemma}
\begin{proof}
Let $\rho$ be a sequence of non-negative integers. Let $\pi$ be the corresponding permutation when $\rho$ is considered as an inversion table after adding the minimum number of trailing zeroes, and assume $\pi \in I_k(321)$ for some $k$. By the same argument as in the proof of Theorem~\ref{ik321} above we can split $\pi \in I_k(321)$ into two subsequences, right-to-left minima and other values, such that both subsequences are increasing. We will call these other values \emph{tall} values. If $\pi_i$ is a right-to-left minimum it is smaller than every element after it, so the corresponding entry $\rho_i$ in the inversion table must be $0$. Since $\pi$ is indecomposable this means $\rho_1 \neq 0$, so $\pi_1$ is tall. Next consider an index $i$ such that $\pi_i$ is tall, and so $\rho_i > 0$. Say it has $\ell$ zeroes following it. Those all correspond to left-to-right minima, and since $\pi_i$ is tall they are all smaller than $\pi_i$. Thus $\pi_i$ is greater than the next $\ell$ values, so $\rho_i \geq \ell$. But if $\pi_i$ is greater than these values and no other values, $\pi$ would be decomposable, so $\ell \leq \rho_i - 1$. By the same argument if $i < j$ with $\rho_i, \rho_j > 0$ are separated by $\ell$ zeroes then those zeroes all correspond to elements who form an inversion with $\pi_i$. Since the tall values are increasing $\pi_i < \pi_j$, so aside from these $\ell$ values $j$ is in all the inversions $i$ is in. Thus $\rho_j \geq \rho_i - \ell$, so $\rho$ satisfies the required constraints.

For the other direction, let $\rho$ be a sequence satisfying the above constraints. We define $\pi = \pi_1 \dots \pi_n$ as in the other case, which we want to show is in $I_k(321)$. It suffices to show that $\pi$ is indecomposable and is the union of two increasing subsequences. We let those subsequences be the right-to-left minima and the tall values. The right-to-left minima are increasing by defintion, so we consider the tall values. The value $\pi_i$ is a right-to-left minimum if and only if $\rho_i = 0$, so the tall values have $\rho_i > 0$. Let $i$ be some index such that $\rho_i > 0$ and $j$ be the next index after $i$ such that $\rho_j > 0$. Suppose there are $k$ zeroes between $\rho_i$ and $\rho_j$, then our condition says that $\rho_j \geq \rho_i - k$. Every zero corresponds to a right-to-left minimum, and hence $\pi_i$ is greater than all those elements. This means we cannot have $\pi_i \geq \pi_j$ since then $\rho_i \geq \rho_j + k + 1$, which is a contradiction. Thus $\pi_i < \pi_j$, so the tall values are in increasing order, which means $\pi$ avoids $321$. 

Suppose $\pi$ is decomposable with $\pi = \tau \oplus \sigma$. Since $\rho_1 \neq 0$ there must be some greatest index $r$ such that $\rho_r > 0$ and $\pi_r$ is in $\tau$. Then $\tau_r$ must be followed by $\tau_r$ zeroes for $\tau$ to be a valid component, breaking the assumption on $\rho$ since $\rho_r = \tau_r$ can have at most $\rho_r - 1$ zeroes after it. Therefore $\pi$ is indecomposable which completes our proof.
\end{proof}

\begin{theorem}\label{321rec}
$|I_k(321)| = a_{k, 1}$ where 
\[a_{n,m} = \begin{cases}
1 & \text{ if } n = 0 \\
\sum_{i = 1}^n a_{n - i, i} & \text{ if } m = 1 \\
a_{n, m - 1} + \sum_{i = m}^n a_{n - i, i} & \text{ otherwise} \\
\end{cases}\]
\end{theorem}

\begin{proof}
In this proof we will consider sequences equal if they differ only in their number of trailing zeroes, since by the same argument as in the proof of Theorem~\ref{partitions} we can uniquely recover the number of trailing zeroes due to the indecomposability condition. We show that $a_{n,1}$ enumerates sequences given by the description in Lemma~\ref{321inv}. More specifically we show that $a_{n,m}$ enumerates prefixes of such sequences with sum $n$ and whose next non-zero element must be at least $m$. Clearly if $n = 0$ there is only one such prefix, the empty prefix. For the other two cases we consider what element could be appended to our prefix to create a longer prefix. If $m = 1$ then our current prefix is either empty or ends with a $1$, so we cannot append a zero. But we can append any positive integer $i \leq n$, which leaves $n - i$ of the sum left, and the next non-zero value must be at least $i$, giving the sum above. In the final case our prefix's last non-zero element is greater than $1$, so we can append a zero. If we append a zero the next non-zero element can be allowed to be one less, so we get $a_{n,m-1}$ for this case. This also prevents us from having more than $x - 1$ zeroes after the value $x$ as we create our prefix. The sum is the same as in the last case, completing our proof.
\end{proof}

Not only can this view help with computing new terms, but it also gives rise to new bijective correspondences. Consider fountains of coins where we only count coins in even rows. We can still place coins as we like, but when tallying the number we only count those in the bottom row, those 2 rows up, 4 rows up and so on. We will call such fountains with $n$ counted coins \emph{even fountains of size $n$}. Bala~\cite{fountain} has tackled the problem of enumerating such even fountains. He shows that they are equinumerous to parallelogram polyominoes by an algebraic argument, and leaves it open to find a bijective proof. Using $I_k(321)$, we give this bijective proof. 

\begin{theorem}\label{parallelo}
There exists a bijective map taking parallelogram polyominoes to even fountains of size $k$.
\end{theorem}
\begin{proof}
Since Theorem~\ref{ik321} is proved by bijection, it suffices to map the inversion tables of elements in $I_k(321)$ bijectively to even fountains of size $k$. 

Write out the fountain in the usual manner (see Figure~\ref{figone}), calling coins in even rows red and the ones in odd rows black for convenience. For each coin in the bottom row, going from left to right, we do the following procedure repeatedly:

\begin{itemize}
\item If we are on a red coin we remove it and move to the coin above and to the right. If there is no such coin we stop.
\item If we are on a black coin we remove it and move to the coin above and to the right. If there is no such coin we move to the coin below and to the right instead.
\end{itemize}

Once done with a coin in the bottom row we write down the number of red coins removed during this procedure. This produces a sequence of numbers, which we then append a single zero to. Call the result $\rho$. We claim this procedure is our desired map.

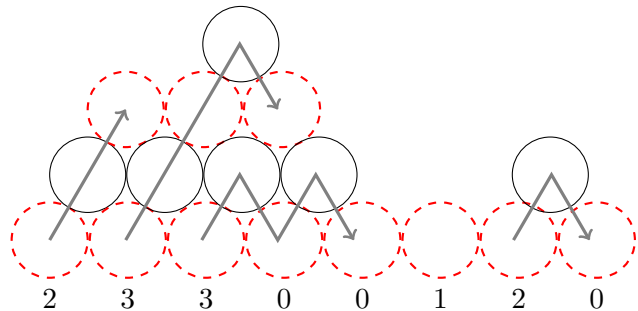
\begin{figure}[h]
\centering
\begin{tikzpicture}
\node[circle,draw,minimum size=1cm, red, dashed, thick](r1c1) {};
\node[circle,draw,minimum size=1cm](r2c1) at (60:1cm) {};
\node[circle,draw,minimum size=1cm, red, dashed, thick](r3c1) at (60:2cm) {};
\node[circle,minimum size=1cm](r4c1) at (60:3cm) {};

\node[circle,draw,minimum size=1cm, right=0mm of r1c1, red, dashed, thick](r1c2) {};
\node[circle,draw,minimum size=1cm, right=0mm of r2c1](r2c2) {};
\node[circle,draw,minimum size=1cm, right=0mm of r3c1, red, dashed, thick](r3c2) {};
\node[circle,draw,minimum size=1cm, right=0mm of r4c1](r4c2) {};

\node[circle,draw,minimum size=1cm, right=0mm of r1c2, red, dashed, thick](r1c3) {};
\node[circle,draw,minimum size=1cm, right=0mm of r2c2](r2c3) {};
\node[circle,draw,minimum size=1cm, right=0mm of r3c2, red, dashed, thick](r3c3) {};

\node[circle,draw,minimum size=1cm, right=0mm of r1c3, red, dashed, thick](r1c4) {};
\node[circle,draw,minimum size=1cm, right=0mm of r2c3](r2c4) {};

\node[circle,draw,minimum size=1cm, right=0mm of r1c4, red, dashed, thick](r1c5) {};
\node[circle,minimum size=1cm, right=0mm of r2c4](r2c5) {};

\node[circle,draw,minimum size=1cm, right=0mm of r1c5, red, dashed, thick](r1c6) {};
\node[circle,minimum size=1cm, right=0mm of r2c5](r2c6) {};

\node[circle,draw,minimum size=1cm, right=0mm of r1c6, red, dashed, thick](r1c7) {};
\node[circle,draw,minimum size=1cm, right=0mm of r2c6](r2c7) {};

\node[circle,draw,minimum size=1cm, right=0mm of r1c7, red, dashed, thick](r1c8) {};

\draw[->, very thick, gray] (0, 0) -- ++(60:1cm) -- ++(60:1cm);
\draw[->, very thick, gray] (1, 0) -- ++(60:1cm) -- ++(60:1cm) -- ++(60:1cm) -- ++(-60:1cm);
\draw[->, very thick, gray] (2, 0) -- ++(60:1cm) -- ++(-60:1cm) -- ++(60:1cm) -- ++(-60:1cm);
\draw[->, very thick, gray] ($(6, 0) + (0.1,0)$) -- ++(60:1cm) -- ++(-60:1cm);

\node[below=0mm of r1c1] (t1) {$2$};
\node[below=0mm of r1c2] (t2) {$3$};
\node[below=0mm of r1c3] (t3) {$3$};
\node[below=0mm of r1c4] (t4) {$0$};
\node[below=0mm of r1c5] (t5) {$0$};
\node[below=0mm of r1c6] (t6) {$1$};
\node[below=0mm of r1c7] (t7) {$2$};
\node[below=0mm of r1c8] (t8) {$0$};
\end{tikzpicture}
\caption{Example of bijective map in Theorem~\ref{parallelo}.}
\label{figone}
\end{figure}

First we prove that $\rho$ is the inversion table of an element in $I_k(321)$. It suffices to show that $\rho$ fits the description given by Lemma~\ref{321inv}. The path for the $i$-th coin from the right can move at most $i - 1$ times upwards and at most $i - 1$ times downwards. Since every other coin encountered is red, this means the resulting value can be at most $i$. Since we append a single zero afterwards $\rho$ is subdiagonal.

For a value $\rho_i$ to be zero, the path corresponding to the next non-zero value $\rho_j = x$ to the left must have touched the $i$-th coin in the bottom row. That path can touch at most $x - 1$ coins in the bottom row aside from where it started, so a value of $x > 0$ can be followed by at most $x - 1$ zeroes. 

Next we prove that if the $i$-th coin and $(i+k+1)$-th coin are separated by $k$ zeroes, where $k$ can be zero, and $\rho_{i+k+1}, \rho_i > 0$ then $\rho_{i+k+1} \geq \rho_i - k$. We start with the case $k > 0$. In this case, the path corresponding to $\rho_i$ must have touched $k$ coins in the bottom row, then continued on from the coin corresponding to $\rho_{i+k}$. Since the path can never cross down below a red coin once above it, the path must have zig-zagged up and down alternatingly until reaching $\rho_{i+k}$. This means we can ignore what is to the left of $\rho_{i+k}$ and consider $\rho_{i+k}$ to have a value of $\rho_i - k$, and refer to the case where $k = 0$. 

For $k = 0$ we have to prove $\rho_i \leq \rho_{i+1}$ for $\rho_i, \rho_{i+1} > 0$. Denote the path corresponding to $\rho_i$ by $P$ and the path corresponding to $\rho_{i+1}$ by $Q$. We will show that for any red coin that $P$ touches aside from the first, $Q$ will touch the black coin below and to the right of it. This means $Q$ touches at least $\rho_i - 1$ black coins, and thus $\rho_{i+1} \geq \rho_i - 1 + 1$, which is what we want to prove. We use induction along the number of steps of the path. For the base case we see that because $\rho_{i+1} \neq 0$, $P$ must start with two upwards steps. Thus $Q$ will start with an upward step as well, so the base case is true. Now for the inductive step $P$ touches some red coin, and $Q$ touches the black coin below and to the right. There are two cases, the first one being that $P$ takes two upward steps next. Since $P$ can never go below that red row again in this case, $Q$ is free to take two upward steps as well, and thus the induction step is complete. In the other case $P$ takes one upward step, then a downward step. Then $P$ ends on a red coin not in the bottom row after these steps. This means $Q$ is free to take a downward step, followed by an upward step since the coins make a valid fountain, completing the induction step. Thus both cases are done, and by induction $\rho_i \leq \rho_{i+1}$.

What remains is to show is that this map is bijective. We note that our map removes a path of coins from the fountain to construct the leftmost value, then applies itself to the remaining fountain to recursively create the remaining suffix. Since the empty fountain is uniquely mapped to the sequence of a single zero, we can use structural induction. We start with injectivity, in which case the problem reduces to showing that there is only one path of any given length you can add onto the left of a fountain. But this can be seen from the fact that as you lengthen the new path, the new coins have uniquely determined positions. A new black coin is always placed above and to the right, and the following red coin can be placed up and to the right if and only if the slot below and to the right is already occupied. Thus the new coins only have one place to go, so the path only has one valid shape for any given length, giving us injectivity. 

This leaves only surjectivity. We note that a single path corresponds to a single value in the inversion table along with a possibly empty run of consecutive zeroes. By Lemma~\ref{321inv} we can see that removing such a sequence of values from the front of an inversion table of a permutation in $I_k(321)$ leaves us with an inversion table in $I_{k - \ell}(321)$ where $\ell$ is the non-zero value removed. Thus we can use the same structural induction as above. Let $\pi \in I_k(321)$, we will show that the lengths of the paths that can be prepended to the corresponding even fountain are exactly the non-zero values that can be prepended to its inversion table $\rho = (\rho_1, \dots, \rho_n)$, along with some number of zeroes. By Lemma~\ref{321inv} we never prepend a single zero, so our path is always non-empty. If we are prepending $\ell$ zeroes then our path will start by touching $\ell$ coins in the bottom row, so we can reduce to the case where our path is $\ell$ steps shorter and starts $\ell$ steps further to the right. Therefore by Lemma~\ref{321inv} we only have to show that any path of length at least $1$ and at most equal to $\rho_1$ can be achieved. We can always remove two coins at a time from the end of a valid new path and have a smaller valid new path, so we only have to show the existence of a path of length $\rho_1$. Let $P$ be the path corresponding to $\rho_1$. Above and to the left of each black coin in $P$ place a red coin. If two adjacent such red coins are at different heights, a black coin will fit between them as that black coin lies to the left of $P$. Otherwise a black coin will fit above them as it lies above $P$. This way we construct a valid path of maximum length. Thus the map is surjective and therefore bijective, completing our proof.
\end{proof}

This leaves us only with $I_k(123)$, the only single pattern giving rise to a sequence not in the OEIS. In accordance with how we define decomposability for permutations, we will consider the sum of two subdiagonal sequences $\rho_1, \dots, \rho_k$ and $\tau_1, \dots, \tau_{\ell}$ to be $\rho_1, \dots, \rho_k, \tau_1 + k, \dots, \tau_{\ell} + k$. As a reminder we say that a sequence of integers of length $n$ is \emph{subdiagonal} if the first value is $\leq n - 1$, the next $\leq n - 2$ and so on. A subdiagonal sequence is then indecomposable if it cannot be written as the sum of two non-empty subdiagonal sequences.

\begin{theorem}\label{123rec}
$|I_k(123)|$ enumerates indecomposable subdiagonal sequences where non-diagonal elements are in decreasing order. Let
\[c_{n,m,k} = 
\begin{cases} 
0 & \text{ if } n < 0 \text{ or } k < 0 \\
1 & \text{ if } n = k = 0 \\
c_{n - 1, m, l - n + 1} + \sum_{i = 0}^{\min(n - 2, m - 1)} c_{n - 1, i, k - i} & \text{ otherwise }
\end{cases}\]
Then the total number of permutations with $k$ inversions (including decomposable ones) avoiding $123$ can be calculated as $\sum_{n = 0}^{k+1} c_{n, n, k}$.

To obtain the number of such permutations that are indecomposable, subtract the $k$-th coefficient of $\p{\sum_{i \geq 0} x^{i(i+1)/2}}^2$.
\end{theorem}

\begin{proof}
First we consider decomposable permutations avoiding $123$. Let $\pi = \tau \oplus \sigma$ be such a permutation. Then $\tau_{\abs{\tau}} \sigma_1$ is an ascent, so neither $\tau$ nor $\sigma$ can contain any other ascents, and are thus decreasing. We also see that the sum of any two decreasing permutations is a decomposable permutation avoiding $123$, so this is a complete description of such permutations. The generating function for decreasing permutations enumerated by inversions is $\sum_{i \geq 0} x^{i(i+1)/2}$ so the composition of two decreasing permutations is enumerated by $\p{\sum_{i \geq 0} x^{i(i+1)/2}}^2$.

We now consider permutations avoiding $123$, including decomposable ones, and prove they are in bijection with subdiagonal sequences where non-diagonal elements are in decreasing order. Let $\pi = \pi_1 \dots \pi_n$ be a permutation with right-to-left maxima at all indices except $i_1, \dots, i_r$. By the same argument as above the values $\pi_{i_1}, \dots, \pi_{i_r}$ are decreasing. Let $\rho = (\rho_1, \dots, \rho_n)$ be the inversion table of $\pi$. Since a right-to-left maximum is greater than all the values that come afterwards, it corresponds to a diagonal value in $\rho$, so $i_1, \dots, i_r$ are the non-diagonal indices of $\rho$. Because $i_{j+1} > i_j$ and $\pi_{i_{j+1}} > \pi_{i_j}$ we get $\rho_{j_{i+1}} > \rho_{j_i}$. Thus the non-diagonal elements of the inversion table are in decreasing order.

For the other direction, let $\rho = (\rho_1, \dots, \rho_n)$ be a subdiagonal sequence with non-diagonal elements in decreasing order and let $\pi$ be the permutation corresponding to $\rho$ when it is considered as an inversion table. We consider the indices of the non-diagonal elements $j_1, \dots, j_s$. The other elements, the right-to-left maxima, are in decreasing order, so it suffices to prove that $\pi_{j_1}, \dots, \pi_{j_s}$ are also in decreasing order, since any occurrence of $123$ would have to involve an ascent among these elements. By assumption we have $\rho_{j_k} > \rho_{j_{k+1}}$ and we also note that any elements between $\pi_{j_k}$ and $\pi_{j_{k+1}}$ are right-to-left maxima. Since $\pi_{j_k}$ is not a right-to-left maximum, it is smaller than all of the elements between $\pi_{j_k}$ and $\pi_{j_{k+1}}$. Thus in order for $\rho_{j_k} > \rho_{j_{k+1}}$ to hold we must have $\pi_{j_k} > \pi_{j_{k+1}}$, so the sequence is decreasing.

All that is left to prove is the formula involving $c_{n,m,k}$. We see that $c_{n,m,k}$ enumerates subdiagonal sequences of $n$ elements with sum $k$ where the non-diagonal elements are decreasing and less than $m$. The three cases $n < 0, k < 0$ and $n = k = 0$ are straight-forward. The last case is covered by considering what the first element of our subdiagonal-sequence is, where $c_{n-1,m,l-n+1}$ is the case where we put a diagonal element first and the sum is when we put a non-diagonal value $i$ in front. Since all elements of a subdiagonal sequence of $n$ elements are less than $n$, we have that $c_{n,n,k}$ counts subdiagonal sequences of $n$ elements with sum $k$ where the non-diagonal elements are decreasing. Thus by Lemma~\ref{invcomp} we only need to sum $n$ from $0$ to $k + 1$ to get our answer.
\end{proof} 

\section{Several patterns}

We now investigate permutations avoiding several patterns. Some groups of patterns are restrictive enough to make all supersets of those patterns trivially determined. For example $|I_k(123, 321)|$ quickly decays to zero by the Erd\H{o}s-Szekeres Theorem~\cite{erds}, making all supersets of $\cp{123, 321}$ easy to determine. We have two more such trivial pairs of patterns.

\begin{theorem}
$|I_k(231, 321)| = 1$ and the unique indecomposable permutation with $k$ inversions avoiding these patterns is $(k+1) 1 2 \dots k$.
\end{theorem}

\begin{proof}
Let $\pi \in I_k(321, 231)$. Since $\pi = \pi_1 \dots \pi_n$ avoids $321$ it is composed of two increasing sequences, the left-to-right maxima and the short values. Suppose $\pi_1 \neq n$. This means we have some inversion $(i, j)$ of $\pi$ where $i \neq 1$. Thus $\pi_i$ is a left-to-right maximum and $\pi_j$ is short. Now we have $\pi_j > \pi_1$ since otherwise $1, i, j$ would be an occurrence of $231$. Therefore $\pi$ is decomposable into the permutation consisting of the first $\pi_1$ values and then the rest. This cannot be, so the assumption $\pi_1 \neq n$ is false, i.e. $\pi_1 = n$. But then since $\pi$ avoids $321$ there can be no inversion after $\pi_1$. Thus $\pi$ is exactly $n 1 2 3 \dots (n-1)$. We also see that this avoids $321$ and $231$, completing the proof.
\end{proof}

\begin{theorem}
$|I_k(231, 312)| = |I_k(12)|$.
\end{theorem}

\begin{proof}
Consider the maximum element $\pi_m$ of $\pi \in I_k(231, 312)$. Since $\pi = \pi_1 \dots \pi_n$ is indecomposable, $\pi_m$ cannot be the last element. Now any element $\pi_i$ after $\pi_{m+1}$ must be smaller than $\pi_{m+1}$ since otherwise $m, m+1, i$ would be an occurrence of $312$. Thus, the values after $\pi_m$ must be in decreasing order. Suppose $\pi_i$ comes before $\pi_m$, then it must be smaller than $\pi_n$ since otherwise $i, m, n$ would be an occurrence of $231$. But this means all values before $\pi_m$ are smaller than the decreasing sequence $\pi_m, \dots, \pi_n$, which contradicts the indecomposability of $\pi$. Thus there is no such $\pi_i$ and $\pi$ must be a decreasing sequence, which avoids $12$. We also see that any decreasing sequence avoids $231$ and $312$, completing the proof.
\end{proof}

All supersets of $\{231, 321\}$ and $\{231, 312\}$ are now trivial to determine, as they all equal to one of the two cases above or decay to zero. 

Up to symmetries the only pairs of patterns left to investigate are $\cp{123, 231}$ and pairs containing $132$.

\begin{theorem}\label{fountainsum}
$|I_k(123, 231)|$ enumerates fountains of $k$ coins where the missing coins with respect to a full triangular fountain form a rectangle (removing no coins counts as a rectangle). The generating function is given by 
\[\sum_{i \geq 1} x^{\binom{i}{2}} + \sum_{i \geq 1} \sum_{j \geq 1} \sum_{\ell = 0}^{\min(i, j) - 1} x^{\binom{i + 1}{2} + \binom{j + 1}{2} - \binom{\ell + 1}{2}}\]
\end{theorem}

\begin{proof}
Let $\pi \in I_k(123, 231)$. Like in the proof of Theorem~\ref{ik231} we can reverse $\pi$ to get $\pi'$ which corresponds to a partition $\rho$. This means $\pi'$ avoids $321$, so by the results for $I_k(123)$ we know that the non-zero entries in the inversion table of $\pi'$ are in weakly increasing order. $\pi'$ also avoids $132$ so the inversion table of $\pi'$ is weakly decreasing as well by Lemma~\ref{inv132}. But therefore, aside from trailing zeroes, the entries of $\rho$ are weakly increasing and weakly decreasing, and are thus fixed. Thus the rows where coins are removed are contiguous and the coins removed in those diagonals is fixed, forming a removed rectangle. For the other direction we see that if anything is removed, it has to include the top coin, so $\rho$ is fixed aside from trailing zeroes. Therefore $\pi'$ avoids $132$ by Lemma~\ref{inv132} and the non-zero entries of $\rho$ are weakly increasing, so $\pi'$ avoids 321 as well. This means $\pi$ avoids $123$ and $231$, completing our proof of the first claim. 

We obtain the generating function by first splitting into cases based on whether the fountain is full or not. If we remove nothing, we are left with full fountains which are enumerated by the first sum in our result. If we remove a non-empty rectangle we are left with two peaks which each form a full fountain with $i> 0$ coins and $j > 0$ coins at the base respectively. When these fountains overlap the overlap takes the shape of another full fountain with $\ell$ coins at the base. The overlap cannot contain either fountain so $\ell < i, j$. Under these constraints any choice of $i, j, \ell$ will give a valid fountain corresponding to an element in $I_k(123, 231)$ with $\binom{i + 1}{2} + \binom{j + 1}{2} - \binom{\ell + 1}{2}$ coins.
\end{proof}

\begin{figure}
\centering
\begin{tikzpicture}
\node[circle,draw,minimum size=1cm,pattern=vertical lines, pattern color=blue](r1c1) at (0,0) {};
\node[circle,draw,minimum size=1cm,pattern=vertical lines, pattern color=blue](r2c1) at (60:1cm) {};
\node[circle,draw,minimum size=1cm,pattern=vertical lines, pattern color=blue](r3c1) at (60:2cm) {};
\node[circle,draw,minimum size=1cm,pattern=vertical lines, pattern color=blue](r4c1) at (60:3cm) {};
\node[circle,draw,minimum size=1cm,pattern=vertical lines, pattern color=blue](r5c1) at (60:4cm) {};
\node[circle,draw,minimum size=1cm, red, dashed, thick](r6c1) at (60:5cm) {};
\node[circle,draw,minimum size=1cm, red, dashed, thick](r7c1) at (60:6cm) {};

\node[circle,draw,minimum size=1cm,pattern=vertical lines, pattern color=blue](r1c2) at ($(r1c1) + (0:1cm)$) {};
\node[circle,draw,minimum size=1cm,pattern=vertical lines, pattern color=blue](r2c2) at ($(r2c1) + (0:1cm)$) {};
\node[circle,draw,minimum size=1cm,pattern=vertical lines, pattern color=blue](r3c2) at ($(r3c1) + (0:1cm)$) {};
\node[circle,draw,minimum size=1cm,pattern=vertical lines, pattern color=blue](r4c2) at ($(r4c1) + (0:1cm)$) {};
\node[circle,draw,minimum size=1cm, red, dashed, thick](r5c2) at ($(r5c1) + (0:1cm)$) {};
\node[circle,draw,minimum size=1cm, red, dashed, thick](r6c2) at ($(r6c1) + (0:1cm)$) {};

\node[circle,draw,minimum size=1cm,pattern=vertical lines, pattern color=blue](r1c3) at ($(r1c2) + (0:1cm)$) {};
\node[circle,draw,minimum size=1cm,pattern=vertical lines, pattern color=blue](r2c3) at ($(r2c2) + (0:1cm)$) {};
\node[circle,draw,minimum size=1cm,pattern=vertical lines, pattern color=blue](r3c3) at ($(r3c2) + (0:1cm)$) {};
\node[circle,draw,minimum size=1cm, red, dashed, thick](r4c3) at ($(r4c2) + (0:1cm)$) {};
\node[circle,draw,minimum size=1cm, red, dashed, thick](r5c3) at ($(r5c2) + (0:1cm)$) {};

\node[circle,draw,minimum size=1cm,pattern=vertical lines, pattern color=blue](r1c4) at ($(r1c3) + (0:1cm)$) {};
\node[circle,draw,minimum size=1cm,pattern=horizontal lines, pattern color=green](r1c4) at ($(r1c3) + (0:1cm)$) {};
\node[circle,draw,minimum size=1cm,pattern=vertical lines, pattern color=blue](r2c4) at ($(r2c3) + (0:1cm)$) {};
\node[circle,draw,minimum size=1cm,pattern=horizontal lines, pattern color=green](r2c4) at ($(r2c3) + (0:1cm)$) {};
\node[circle,draw,minimum size=1cm,pattern=horizontal lines, pattern color=green](r3c4) at ($(r3c3) + (0:1cm)$) {};
\node[circle,draw,minimum size=1cm,pattern=horizontal lines, pattern color=green](r4c4) at ($(r4c3) + (0:1cm)$) {};

\node[circle,draw,minimum size=1cm,pattern=vertical lines, pattern color=blue](r1c5) at ($(r1c4) + (0:1cm)$) {};
\node[circle,draw,minimum size=1cm,pattern=horizontal lines, pattern color=green](r1c5) at ($(r1c4) + (0:1cm)$) {};
\node[circle,draw,minimum size=1cm,pattern=horizontal lines, pattern color=green](r2c5) at ($(r2c4) + (0:1cm)$) {};
\node[circle,draw,minimum size=1cm,pattern=horizontal lines, pattern color=green](r3c5) at ($(r3c4) + (0:1cm)$) {};

\node[circle,draw,minimum size=1cm,pattern=horizontal lines, pattern color=green](r1c6) at ($(r1c5) + (0:1cm)$) {};
\node[circle,draw,minimum size=1cm,pattern=horizontal lines, pattern color=green](r2c6) at ($(r2c5) + (0:1cm)$) {};

\node[circle,draw,minimum size=1cm,pattern=horizontal lines, pattern color=green](r1c7) at ($(r1c6) + (0:1cm)$) {};
\end{tikzpicture}
\label{triangles}
\caption{Example of decomposition into smaller fountains.}
\end{figure}

This leaves us with four pairs, pairing $132$ with any of the patterns $123, 213, 231$ and $321$. We now tackle them in that order. To deal with the first of them we have to introduce a new kind of partition.

\emph{Almost triangular} partitions of $k$ are partitions $\rho = (\rho_1, \dots, \rho_r)$ of $k$ such that each $\rho_i$ is either $i$ or $i - 1$, allowing $\rho_1 = 0$ for convenience. We note however that $0+1+\dots+r$ and $1+\dots+r$ are the same partition, so since we allow $\rho_1 = 0$ we must forbid $\rho_i = i - 1$ being true for all $i$ to avoid duplicates.

\begin{theorem}\label{almosttriangular}
$|I_k(132, 123)|$ enumerates the Pascal triangle with the first column removed, the corresponding sequence is entry \texttt{A135278} in the OEIS. Its generating function is 
\[\sum_{n \geq 1} x^{(n-2)(n+1)/2} \p{(x + 1)^n - 1}\].
\end{theorem}

\begin{proof}
We first show that almost triangular partitions of $k$ are enumerated by the Pascal triangle with the first column removed. Suppose we have an almost triangular partition $\rho = (\rho_1, \dots, \rho_r)$ of $k$. The right hand side is at most $1 + \dots + r = r(r+1)/2$ and at least $0 + 1 + \dots + (r - 1) = r(r-1)/2$.  Therefore the sum is at least $r(r-1)/2+1$. This means that for any given $k$, the $r$ is uniquely determined as the index of the next triangular number after or equal to $k$. To create an almost triangular partition we can then start with the least possible value for each of the $r$ summands, then choose $k - r(r - 1)/2$ of them to be increased by one. Since we must always increase at least one summand there will be $\binom{r}{k - r(r-1)/2}$ ways to do this, which walks through the Pascal triangle row by row without the first column as $k$ increases.

Now we show that $I_k(132, 123)$ corresponds bijectively to almost triangular partitions. Let $\pi \in I_k(132, 123)$. For some index $i$ suppose at least two values after $\pi_i$ are not less than $\pi_i$, say $\pi_j, \pi_k$ with $j < k$. Then either $\pi_j > \pi_k$ and $i, j, k$ is an occurrence of $132$ or $\pi_j < \pi_k$ and $i, j, k$ is an occurrence of $123$. Thus each $\pi_i$ is greater than all of the values after it or all but one, proving that $\rho$ is almost triangular. For the other direction we now assume we have $\pi \in I_k(132)$ such that $\rho$ is almost triangular. Then each value in $\pi$ must be greater than every value after it but one. Thus we can have no occurrence of $123$ since that requires a value to have two greater values after it. 

A row of the Pascal triangle has generating function $(x + 1)^n$, so to delete the first column we simply subtract $1$. Simple counting gets us that this row will start at $(n-2)(n+1)/2$, so by simply shifting over each row the appropriate amount and summing them together shows that this generating function corresponds to our description.
\end{proof}

Our next result involves a kind of partition called a Gorenstein partition. \emph{Gorenstein partitions} are partitions whose corresponding Schubert variety has a Gorenstein homogeneous coordinate ring, see Svanes~\cite{svanes} and Stanley~\cite{stanley}. The terminology appears to be due to Sloane~\cite{oeis} under A117629 however. By Theorem 5.4 in Stanley's work~\cite{stanley} the definition is equivalent to partitions whose maximal chains are all of the same size when regarded as order ideals of $\{1, 2, \dots\} \times \{1, 2, \dots\}$ under the product order.  See Figure~\ref{hasse} which shows a combined Hasse and Ferrers diagram for such a partition and the corresponding partial order. The figure is rotated by $45^{\circ}$ compared to a normal Ferrers diagram to conform to Hasse diagram conventions. The maximal chains all have size seven in this example, and a couple of them have been highlighted in the figure. This definition is rather unwieldy for our purposes however, so we first translate this condition.

\begin{figure}
\centering
\begin{tikzpicture}
\node (P) at (1.5, 0)  {\large $24 = 7 + 4 + 4 + 4 + 2 + 2 + 1$};
\begin{scope}[rotate=-45,shift={(2,0)}]
\foreach \x/\y in {0/0,1/0,2/0,3/0,4/0,5/0,6/0,0/-1,1/-1,2/-1,3/-1,0/-2,1/-2,2/-2,3/-2,0/-3,1/-3,2/-3,3/-3,0/-4,1/-4,0/-5,1/-5,0/-6} {
	\draw[fill=lightgray] (\x,\y) circle (0.25);
	\ifthenelse{\x=0}{}{\draw[thick] (\x-0.25, \y) -- (\x - 0.75, \y)};
	\ifthenelse{\y=0}{}{\draw[thick] (\x, \y+0.25) -- (\x, \y + 0.75)};
}

\draw[dashed] (-0.5, 0.5) -- (-0.5, -6.5) -- (0.5, -6.5) -- (0.5, 0.5) -- cycle;
%\draw[dashed] (-0.5, 0.5) -- (6.5, 0.5) -- (6.5, -0.5) -- (-0.5, -0.5) -- cycle;
\draw[dashed] (-0.5, 0.5) -- (2.5, 0.5) -- (2.5, -2.5) -- (3.5, -2.5) -- (3.5, -3.5) -- (1.5, -3.5) -- (1.5, -0.5) -- (-0.5, -0.5) -- cycle;
\end{scope}
\end{tikzpicture}
\caption{Hasse/Ferrers diagram of a Gorenstein partition.}
\label{hasse}
\end{figure}

\begin{lemma}
\label{gorenlemma}
A partition $\rho$ is Gorenstein if and only if $\rho_i + i$ is constant across indices $i$ that satisfy $\rho_i \neq \rho_{i+1}$. 
\end{lemma}

\begin{proof}
Let $\rho_1, \dots, \rho_r$ be a partition of $n$. Let us assume $j$ is the maximal index such that $\rho_j = \rho_1$. Since $\rho_j \neq \rho_{j+1}$ the chain from $(1, 1)$ to $(\rho_1, j)$ is maximal since neither the point right of it or below it are in the order ideal corresponding to the partition. Then this maximal chain length must be $\rho_1 + j - 1$. The same argument holds for any $\rho_i$ such that $\rho_i \neq \rho_{i+1}$ and the chain length they give is $\rho_i + i - 1$. However, if $\rho_i = \rho_{i+1}$ then one can always go further, making the chain from $(1, 1)$ to $(\rho_i, i)$ not maximal. Therefore all maximal chains are of the same size if and only if $\rho_i + i$ is constant for all $i$ such that $\rho_i \neq \rho_{i+1}$.
\end{proof}

\begin{theorem}
$|I_k(132, 213)|$ is the number of Gorenstein partitions of $k$. The corresponding sequence is entry \texttt{A117629} in the OEIS.
\end{theorem}

\begin{proof}
Let $\pi \in I_k(132)$ and $\rho$ be the partition corresponding to its inversion table as given by Theorem~\ref{partitions}. We wish to show that if $\pi \in I_k(213)$ then $\rho$ is Gorenstein. Suppose we have two indices $i < j$ such that $\rho_i > \rho_{i + 1}$ and $\rho_j > \rho_{j + 1}$. For $\rho_i > \rho_{i + 1}$ to hold we must have $\pi_i > \pi_{i+1}$ and similarly $\pi_j > \pi_{j + 1}$. Then there cannot be any index $i < k < j$ such that $\pi_i < \pi_k$ since then $i, i + 1, k$ would be an occurrence of $213$. Neither can there be an index $k > j$ such that $\pi_j < \pi_k$ since then $j, j + 1, k$ would be an occurrence $213$. Thus any values contributing to the inversion count $\rho_i$ but not $\rho_j$ must occur between the indices $i$ and $j$, and conversely every value there between must contribute to that difference. Thus we arrive at $\rho_i - \rho_j = j - i$ or equivalently $\rho_i + i = \rho_j + j$. Thus $\rho$ is Gorenstein by Lemma~\ref{gorenlemma}.

For the other direction assume $\rho$ is Gorenstein. If $\rho$ contains only one unique non-zero value then $\pi$ must be of the form $k, k + 1, \dots, n, 1, 2, \dots, k - 1$ for some $k, n$. In this case $\pi$ clearly has no occurrence of $213$, so we can assume $\rho$ has at least two different non-zero values. We proceed by contradiction so assume $\pi = \pi_1 \dots \pi_n$ has an occurrence of $213$ at indices $a, b, c$. Without loss of generality we can choose the values $a, b, c$ such that $a$ is maximal among all such choices. Next we show we can assume $b = a + 1$ by showing that if $b > a + 1$ there exists a smaller valid choice of $b$. Suppose $b \neq a + 1$ so there is some $a < k < b$. We cannot have $\pi_k > \pi_c$ since then $a, k, c$ would be an occurrence of $132$. If $\pi_k < \pi_a$ then we can replace $b$ with $k$. The only case left then is $\pi_a < \pi_k < \pi_c$, but in that case we can replace $a$ with $k$ which contradicts the maximality of $a$. Thus we can assume $b = a + 1$ and still have $a$ maximal. Now $\rho_a > \rho_{a+1}$. Since $\rho$ has more than one non-zero value there must be some other index $i$ such that $\rho_i > \rho_{i+1}$, where we define $\rho_j = 0$ for $j > \abs{\rho}$. We now split into two cases.

We consider the case $a < i$ and can assume this $i$ to be minimal among such $i$. This means $\rho_{a+1} = \rho_i$, so $\pi_{a+1}, \dots, \pi_i$ is increasing. Suppose that $\pi_a < \pi_i$. Then the only inversions contributing to $\rho_a$ but not $\rho_i$ are of the form $(a, j)$ with $j < i$. But there can be at most $i - a - 1$ such inversions and there are only $i - a - 1$ indices between $i$ and $a$. But then $\rho_a - \rho_i < i - a$, which contradicts the fact that $\rho$ is Gorenstein by Lemma~\ref{gorenlemma}. Thus we can assume $\pi_a > \pi_i$. Hence the values $\pi_{a+1}, \dots,\pi_i$ are all less than $\pi_a$, so we must have $i < c$. But then $i, i + 1, c$ is an occurrence of $213$ with a larger value of $a$, which gives us a contradiction that completes this case.

Now we assume there is no such $i > a$, so $\rho_a$ is the last non-zero value which makes $\rho_{a+1} = 0$. This means $\pi_{a+1}, \dots, \pi_n$ is increasing, so we can assume $c = n$. Since $\pi$ is indecomposable there is some index $j < a$ such that $\pi_j > \pi_n$. Then $\pi_j > \pi_a$ and $j < a$ so we must have $\rho_a < \rho_j$. We can now let $j$ be the maximum index such that $\pi_j > \pi_n$ and $\rho_j > \rho_a$. Then $\rho_j > \rho_{j+1}$ since otherwise $j + 1$ would be a valid greater index. Now since $j < a$ there is some minimum index $\ell$ such that $\rho_j > \rho_\ell > \rho_{\ell+1}$. Since $\ell$ was not a valid choice for $j$ we must have $\pi_j > \pi_n > \pi_{\ell}$. Since $\ell$ is minimal we have that $\rho_{j+1} = \rho_\ell$, so $\pi_{j+1}, \dots, \pi_\ell$ is increasing. Thus $\pi_j$ is greater than all of these values. Therefore we get that $\rho_j - \rho_{\ell} > \ell - j$ which contradicts the fact that $\rho$ is Gorenstein by Lemma~\ref{gorenlemma}. This completes our proof.
\end{proof}

\begin{theorem}
Let $\mu \vDash s$ denote that $\mu$ is a composition of $s$. $|I_k(132, 213)|$ has generating function 
\[\sum_{s \geq 0} \sum_{\mu \vDash s, \abs{\mu} \neq 1} x^{\binom{s}{2} - \sum_{m \in \mu} \binom{m}{2}}\]
Thus, $|I_k(132, 213)|$ also enumerates finite sequences of positive integers of length $> 1$ such that $k$ equals the second elementary symmetric function of the values of the sequence, as noted in the OEIS entry \texttt{A117629}.
\end{theorem}

\begin{proof}
For a Gorenstein parititon $\rho$ let $i$ be an index such that $\rho_i \neq \rho_{i+1}$, we then define $\rho_i + i$ as the \emph{diagonal constant} of $\rho$, as it is the same for all such $i$ by Lemma~\ref{gorenlemma}. We consider some fixed index $s$ in the generating function sum. The case $s = 0$ is obvious, so we can assume $s > 0$. Let $\rho$ be a Gorenstein partition with diagonal constant $s$. Let $i_1 < \dots < i_r$ be the indices of $\rho$ such that $\rho_{i_j} \neq \rho_{i_j+1}$. We then prepend $i_0 = 0$ and append $i_{r+1} = s$ to the sequence. Then $\sum_{j = 1}^{r+1} i_j - i_{j-1} = s$, so $\mu = i_1 - i_0, i_2 - i_1, \dots, i_{r+1} - i_r$ is a composition of $s$ with at least two terms. We also see that any such composition with at least two terms can be transformed back into a valid Gorenstein partition by reversing the steps.

Next we consider what the sum of $\rho$ is in terms of $s$ and $\mu$. We compare $\rho$ to the triangular partition $\sigma = s - 1, s - 2, \dots, 1$ which has sum $\binom{s}{2}$ and is Gorenstein. We see that $\sigma_1 + 1 = s$ so $\sigma_{i_j} + i_j = \rho_{i_j} + i_j$ for all $j$ since $\sigma_{\ell} \neq \sigma_{\ell + 1}$ for all $\ell$. If $\mu_j = x$ this means $\rho$ contains $x$ equal consecutive values ending in $\rho_{i_j} = \sigma_{i_j}$. Thus the sum of those $x$ values is $\binom{x}{2}$ lower than the sum of the corresponding values in $\sigma$. Taking the sum over all the values of $\mu$ we get that the sum of $\rho$ is $\binom{s}{2} - \sum_{m \in \mu} \binom{m}{2}$, which completes our proof for the generating function.

The relation to symmetric functions can be seen by expanding the exponent 
\[\binom{\sum_{m \in \mu} m}{2} - \sum_{m \in \mu} \binom{m}{2}\]
and seeing that the pure terms $m, m^2$ cancel and the mixed terms add to cancel the $2$ in the denominator of the binomial.
\end{proof}

The generating function $\sum_{s \geq 0} \sum_{\mu \vDash s, \abs{\mu} \neq 1} x^{\binom{s}{2} - \sum_{m \in \mu} \binom{m}{2}}$ is not useful for actually computing new terms in the sequence, but we can use the following recurrence instead. By ignoring all but the first and last $\sqrt{n}$ summands in the recurrence below, as they are zero, we can compute the $n$-th value in $\mathcal{O}(n^{2.5})$ time and $\mathcal{O}(n^2)$ space.

\begin{theorem}
The number of Gorenstein partitions with sum $n$, and thus also the number of elements in $|I_n(132, 213)|$, is given by the sum $\sum_{d = 0}^n f(n, d)$ where
\[f(n, d) = \begin{cases}
0 & \text{ if } n < 0 \\
1 & \text{ if } n = 0  \\
\sum_{k = 1}^d f(n - k(d + 1 - k), d - k) & \text{ otherwise }
\end{cases}\]
\end{theorem}

\begin{proof}
We claim that $f(n, d)$ equals the number of Gorenstein partitions with sum $n$ and diagonal constant $d + 1$, from which the result would follow.

The cases with $n \leq 0$ are trivial, so we focus on the last one. Say we have some Gorenstein partition $\rho_1, \dots, \rho_r$ with sum $n$ and diagonal constant $d + 1$. Then there is some maximal index $k$ such that $\rho_1 = \rho_k$ and $\rho_k \neq \rho_{k+1}$. Then we must have $\rho_k + k = d + 1$. Thus $\rho_1 = \rho_2 = \dots = \rho_k = d + 1 - k$, so the sum of these entries is $k(d + 1 - k)$. If we consider the remaining entries unto themselves, assuming there are any, they must form another Gorenstein partition with sum $n - k(d + 1 - k)$ and diagonal constant $d + 1 - k$, since all entries are shifted left by $k$. Summing over all such $k$ we get the formula above.
\end{proof}

\begin{theorem}
$|I_k(132, 231)|$ is equal to the number of partitions on $k$ elements into distinct parts, the corresponding sequence is entry \texttt{A000009} in the OEIS.
\end{theorem}

\begin{proof} 
Let $\pi$ be a permutation and $\rho$ be the partition corresponding to its inversion table as given by Theorem~\ref{partitions}. It suffices to show that an occurrence of $231$ in $\pi$ is equivalent to $\rho$ containing two equal values. Since $\rho$ is written in decreasing order this is equivalent to $\rho$ containing two equal adjacent values, say $\rho_j = \rho_{j+1}$. We cannot have $\pi_j > \pi_{j + 1}$ since then we would have $\rho_j > \rho_{j+1}$, so $\pi_j < \pi_{j+1}$. Since $\rho_{j+1} > 0$ we must have some $\ell$ such that $\pi_{j+1} > \pi_{\ell}$ and $j + 1 < \ell$. But then $j, j + 1, \ell$ is an occurrence of the pattern $231$.

For the other direction suppose we have a $132$-avoiding permutation $\pi = \pi_1 \dots \pi_n$ that has an occurrence of the pattern $231$ at the indices $t, u, v$. Without loss of generality take $t$ to be the maximal valid index $t$ for the given $u, v$. Suppose there is an index $t < i < u$ such that $\pi_t > \pi_i$. Since $t$ is maximal this means $\pi_i < \pi_v$, otherwise $i$ would be a valid choice for the index $t$. But then $i, u, v$ is an occurrence of the pattern $132$, which cannot be. Thus there is no such index $i$. Therefore any inversion $(t, j)$ gives rise to an inversion $(u, j)$. This means $\rho_u \leq \rho_t$. But since the inversion table must be weakly decreasing $\rho_u \geq \rho_t$, so $\rho_u = \rho_t$ which completes our proof.
\end{proof}

\begin{theorem}
$|I_k(132, 321)|$ is equal to the number of partitions on $k$ elements into equal values. This is in turn equal to the number of divisors of $k$. The corresponding sequence is entry \texttt{A000005} in the OEIS.
\end{theorem}

\begin{proof}
In a similar manner to the previous proof, all we have to show is the equivalence of the partition containing two different values and the permutation having an occurrence of the pattern $321$. Suppose first that our partition $\rho = (\rho_1, \dots, \rho_r)$ contains two unequal values and let $\pi = \pi_1 \dots \pi_n$ be the permutation that has $\rho$ has an inversion table after appending the minimum number of required zeros to make it subdiagonal. Since the partition is written in decreasing order we can without loss of generality assume these values are adjacent, say $\rho_j \neq \rho_{j+1}$. Since the partition is weakly decreasing we must then have $\rho_j > \rho_{j + 1}$. This means $\pi_j > \pi_{j+1}$. Since $\rho_{j + 1} > 0$ we must then have some $\ell$ such that $\pi_{j + 1} > \pi_{\ell}$ and $j + 1 < \ell$. But then $j, j + 1, \ell$ is an occurrence of the pattern $321$.

For the other direction we assume we have a $132$-avoiding permutation $\pi = \pi_1 \dots \pi_n$ that has an occurrence of the pattern $321$ at $t, u, v$. We let $\rho = (\rho_1, \dots, \rho_r)$ be the non-zero entries of the inversion table of $\pi$. Then any inversion $(u, w)$ gives us an inversion $(t, w)$ so $\rho_t \geq \rho_u$. Furthermore $\pi_t > \pi_u$ so $\rho_t > \rho_u$. Since $\pi_v < \pi_u$ we also have $\rho_u > 0$. Thus the partition contains two different positive values, completing our proof.
\end{proof}

\section{More than two patterns}

Most of the remaining pattern combinations are trivially deduced as some subset of the patterns forces the sequence to die out or contain only very specific permutations, as noted before. We consider here the complement of those cases.

\begin{theorem}
$|I_k(123, 132, 231)| = 1$.
\end{theorem}

\begin{proof}
We know that $I_k(123, 132)$ gives us almost triangular partitions and that $I_k(132, 231)$ gives us partitions with distinct parts, so this gives us almost triangular partitions with distinct parts. The number of values in an almost triangular partition of $k$ is uniquely determined as noted in Theorem~\ref{almosttriangular}. Suppose we start with the partition $0, 1, 2, \dots, r$. Incrementing any subset of $\ell > 0$ values by one produces a duplicate, except if we increment exactly the last $\ell$ values. This is because if we do not choose the last $\ell$ values there will be some value that is increased by one, followed by a value that is unchanged, producing a duplicate value. Thus if $1 + 2 + \dots + r = k - \ell$ we have to increment the last $\ell$ values, producing a unique partition.
\end{proof}

\begin{theorem}\label{diagonalpascal}
$|I_k(123, 132, 213)|$ enumerates the Pascal triangle, read by diagonals, offset by two elements. In other words, it reads the binomials $\binom{n}{m}$ in increasing order by the sum $n + m$, with each set being read in increasing order by $n$, starting at $\binom{1}{1}$ for $k = 0$. Aside from $\binom{1}{1}$ it walks over all values with $n \geq 2$, including zeroes, see Figure~\ref{weirdwalk}. Furthermore its generating function can be written as
\[1 + x + \sum_{d \geq 3} x^{1+
\binom{d - 1}{2}} \sum_{n = 2}^d \binom{n}{d - n} x^{n - 2}\]
\end{theorem}

\makeatletter
\newcommand\binomialCoefficient[2]{%
    % Store values 
    \c@pgf@counta=#1% n
    \c@pgf@countb=#2% k
    %
    % Take advantage of symmetry if k > n - k
    \c@pgf@countc=\c@pgf@counta%
    \advance\c@pgf@countc by-\c@pgf@countb%
    \ifnum\c@pgf@countb>\c@pgf@countc%
        \c@pgf@countb=\c@pgf@countc%
    \fi%
    %
    % Recursively compute the coefficients
    \c@pgf@countc=1% will hold the result
    \c@pgf@countd=0% counter
    \pgfmathloop% c -> c*(n-i)/(i+1) for i=0,...,k-1
        \ifnum\c@pgf@countd<\c@pgf@countb%
        \multiply\c@pgf@countc by\c@pgf@counta%
        \advance\c@pgf@counta by-1%
        \advance\c@pgf@countd by1%
        \divide\c@pgf@countc by\c@pgf@countd%
    \repeatpgfmathloop%
    \the\c@pgf@countc%
}
\makeatother

\begin{figure}
\centering
\begin{tikzpicture}
\def\dy{-0.75}
\def\dx{1}
\def\of{-0.5}
\foreach \r in {0,...,8} {
	\foreach \c in {0,...,\r} {
		\node (C-\r-\c) at ({\dx*\c+\of*\r},{\dy*\r}) {
			$\binomialCoefficient{\r}{\c}$
		};
		\ifthenelse{\c=0 \AND \r<2}{}{\node[circle, inner sep=0pt, minimum size=16pt] (H-\r-\c) at (C-\r-\c) {};};
	}
}

\foreach \r/\c in {2/3,2/4,2/5,2/6,3/4,3/5} {
	\node[text=gray] (C-\r-\c) at ({\dx*\c+\of*\r},{\dy*\r}) {
		$0$
	};
	\node[circle, inner sep=0pt, minimum size=16pt] (H-\r-\c) at (C-\r-\c) {};
}

\draw[->, very thick, dashed, red] (H-1-1) -- (H-2-0);
\draw[->, very thick, dashed, red] (H-2-1) -- (H-3-0);
\draw[->, very thick, dashed, red] (H-2-2) -- (H-3-1) -- (H-4-0);
\draw[->, very thick, dashed, red] (H-2-3) -- (H-3-2) -- (H-4-1) -- (H-5-0);
\draw[->, very thick, dashed, red] (H-2-4) -- (H-3-3) -- (H-4-2) -- (H-5-1) -- (H-6-0);
\draw[->, very thick, dashed, red] (H-2-5) -- (H-3-4) -- (H-4-3) -- (H-5-2) -- (H-6-1)  -- (H-7-0);
\draw[->, very thick, dashed, red] (H-2-6) -- (H-3-5) -- (H-4-4) -- (H-5-3) -- (H-6-2)  -- (H-7-1) -- (H-8-0);

\draw[->, thick, blue] (H-2-0) -- (H-2-1);
\draw[->, thick, blue] (H-3-0) to[out=0, in=180, looseness=0.5] (H-2-2);

\foreach \x/\y in {4/3,5/4,6/5,7/6} {
\draw[thick, blue] (H-\x-0) to[out=0, in=210, looseness=0.5] ($(H-\x-0)+(1,0.25)$);
\draw[thick, blue, opacity=0.2, dashed] ($(H-\x-0)+(1,0.25)$) -- ($(-1,-0.25)+(H-2-\y)$);
\draw[->, thick, blue] ($(-1,-0.25)+(H-2-\y)$) to[in=180, out=30, looseness=0.5] (H-2-\y);
}

\node[text=red, rotate=135] (D) at (3.5,-3.25) {\Huge $\dots$};
\end{tikzpicture}
\caption{Enumeration order of Pascal Triangle in Theorem~\ref{diagonalpascal}.}
\label{weirdwalk}
\end{figure}

\begin{proof}
We know that $I_k(123, 132)$ gives us almost triangular partitions and that $I_k(132, 213)$ gives us Gorenstein partitions, so the intersection gives us almost triangular Gorenstein partitions. We show that this is equivalent to considering all partitions given by the following construction. Starting with $r - 1, \dots, 1, 0$ we construct our partition by selecting some subset of the partition that contains at least one out of every two adjacent elements, and then increment every selected value by $1$.

This construction clearly gives us almost triangular partitions, so it suffices to show that failing to increment two adjacent values is equivalent to the partition not being Gorenstein. Assume then we have two indices $a, a + 1$ that do not get incremented. There is always at least one value that gets incremented, so we have some index $b$ that gets incremented. Now $\rho_b = r - b + 1$ and depending on whether or not $b + 1$ gets incremented we have either $\rho_{b+1} = r - b - 1$ or $\rho_{b+1} = r - b$. In either case we have $\rho_b \neq \rho_{b+1}$. Furthermore $\rho_a = r - a \neq r - a - 1 = \rho_{a+1}$. Since $\rho_a + a = r \neq r + 1 = \rho_b + b$ we get that $\rho$ is not Gorenstein by Lemma~\ref{gorenlemma}.

Now we assume $\rho$ is not Gorenstein and wish to show there are two adjacent indices that fail to get incremented. By Lemma~\ref{gorenlemma} we have indices $a < b$ such that $\rho_a \neq \rho_{a+1}$, $\rho_b \neq  \rho_{b+1}$ and $\rho_a + a \neq \rho_b + b$ . If we had incremented index $a + 1$ but not $a$ we would have $\rho_a = \rho_{a+1}$, so we get that if index $a$ was not incremented, index $a + 1$ was not either. But if neither of them was incremented, we already have two adjacent elements neither of which was incremented, which completes this case. Thus we can assume this is not the case, so index $a$ was incremented. The same applies to $b$. But then $\rho_a + a = \rho_b + b$, which is a contradiction, so we are done.

Lastly we show that this construction gives us the claimed enumeration. Let $r$ be the number of elements in the partition, which then also determines how many elements we have to increment, given by $s = k - r(r - 1)/2$. We can now view the construction as choosing the sizes of the consecutive runs of elements we increment with one non-incremented element between each run. The first and last run can be empty, but the others must have at least one element, so this is given by
\[[x^s]\p{\frac{1}{1 - x} \p{\frac{x}{1 - x}}^{r - s - 1} \frac{1}{1 - x}} = [x^{2s + 1 - r}]\frac{1}{(1 - x)^{r - s+ 1}} = \binom{s+1}{r-s+1}\]
We see now that for $r$ fixed, the sum of the arguments in the binomial coefficient is fixed. Furthermore as $s$ increases the upper argument increases as well, so this enumerates the binomial coefficients in the order claimed above. Furthermore for $k = 0$ we have $r = s = 0$, giving the coefficient $\binom{1}{1}$, so the offset is correct as well.

We take $k = 0, 1$ aside especially in the generating function with the $1+x$ at the start. We let $d$ be the sum of the arguments to the binomial, which is at least $3$ for $k \geq 1$. Then we need to sum over all such binomials which appear in the generating function. For $k \geq 2$ we increment some element in the procedure above so $s \geq 1$, so the upper argument is always at least $2$ and at most $d$. Thus the generating function for one fixed $d$ is the inner summand given in the generating function above. Finally we just have to count how many entries appear before a given value of $d$ starts, which comes out to $1+\binom{d - 1}{2}$, completing our proof.
\end{proof}

\begin{theorem}
$|I_k(132, 213, 231)|$ is equal to the number of odd divisors of $k$, the corresponding sequence is entry \texttt{A001227} in the OEIS.
\end{theorem}

\begin{proof}
The number of odd divisors of $k$ is the same as the number of ways to write $k$ as the difference of two triangular numbers, as noted by Mason~\cite{cons}. We know that $I_k(132,231)$ gives us partitions with distinct parts and $I_k(132,213)$ gives us Gorenstein partitions, so the intersection gives us Gorenstein partitions with distinct parts. 

Using our characterisation of Gorenstein partitions we note that since $\rho_i \neq \rho_{i+1}$ for all $i$, we must have that $\rho_i + i$ is constant across the entire partition. Thus if $\rho_i, \rho_{i+1} > 0$ we must have $\rho_i = \rho_{i+1} + 1$ since the partition is weakly decreasing. Therefore Gorenstein partitions with distinct parts are equivalent to writing $k$ as a sum of a set of consecutive run of integers. This is equivalent to writing $k$ as the difference of two triangular numbers by considering the maximum included integer in the sum and minimum excluded integer in the sum. 
\end{proof}

\begin{theorem}
$|I_k(132, 213, 321)| = |I_k(132, 321)|$.
\end{theorem}

\begin{proof}
Inclusion in one direction is obvious, so we simply show that an occurrence of $213$ in $\pi$ implies an occurrence of either $132$ or $231$. Let $a, b, c$ be the indices at which $213$ occurs in $\pi$. If $\pi_k > \pi_c$ for all $k > c$ then $\pi$ would be decomposable, so there exists some $k > c$ such that $\pi_k < \pi_c$. Now if $\pi_k > \pi_b$ then $b, c, k$ would be an occurrence of $132$, in which case we are done. Thus we can assume $\pi_k < \pi_b$. But then $b, c, k$ is an occurrence of $231$, which completes our proof.
\end{proof}

\begin{theorem}
$|I_k(123, 132, 213, 231)|$ enumerates the Pascal triangle, except all values $> 1$ are replaced by $0$. The corresponding sequence is entry \texttt{A103451} in the OEIS and has the generating function 
\[\sum_{i\geq 0} x^{i(i+1)/2} + x^{(i+1)(i+4)/2}\]
\end{theorem}

\begin{proof}
We already know the unique partition $I_k(123, 132, 231)$ gives us for a given $k$. Thus we only have to check for what $k$ this is in $I_k(132, 213, 231)$, i.e. when it is a run of consecutive integers. The smallest non-zero value in an almost triangular partition is always $1$ or $2$, so this is a consecutive run of integers starting with $1$ or $2$. That always gets us a sum of the form $n(n+1)/2$ or $n(n+1)/2 - 1$. This is exactly the indices of ones in the Pascal triangle when read row by row, completing our proof. The generating function follows directly as the first summand enumerates the indices of the leftmost column in the Pascal triangle, including zero, and the other the rightmost column, excluding zero.
\end{proof}

\acknowledgements{

The algorithm to generate permutations with a fixed number of inversions by Effler and Ruskey~\cite{invalg} was used to generate elements of all the sequences above, which helped tremendously in finding the formulas and other results in this paper.

Thanks to Anders Claesson, Henning Úlfarsson, Christian Bean and Jay Pantone both for bringing this topic to my attention and for enlightening discussions on it.

}

\bibliographystyle{acm}

\bibliography{enumeratedbyinversions}

\end{document}